\documentclass{article}

\usepackage{amssymb,amsmath,amsthm,enumitem}

\setlist{nolistsep}        

\newtheorem{theorem}{Theorem}[section]
\newtheorem{lemma}[theorem]{Lemma}
\theoremstyle{definition}
\newtheorem*{examples}{Examples}

\newcommand{\abs}[1]{\left\lvert#1\right\rvert}
\newcommand{\act}{\cdot}
\newcommand{\action}{\curvearrowright}
\newcommand{\Aset}{\mathcal{A}}
\newcommand{\BLipb}{\mathsf{BLip_b}}

\newcommand{\Cb}{\mathsf{C_b}}
\newcommand{\conv}{\star}
\newcommand{\ellinfty}{\ell_\infty}
\newcommand{\Fset}{\mathcal{F}}         
\newcommand{\iin}{\mathord\in}
\newcommand{\mmeas}{\mathfrak{m}}
\newcommand{\Means}{\mathfrak{M}^{+1}(S)}
\newcommand{\Mol}{\mathsf{Mol}}
\newcommand{\norm}[1]{\lVert#1\rVert}
\newcommand{\pmass}{\partial}           
\newcommand{\psm}{\Delta}
\newcommand{\ru}{\mathsf{r}}
\newcommand{\tfine}{\mathsf{\scriptstyle F}}
\newcommand{\Ub}{\mathsf{U_b}}

\newcommand{\wstar}{weak$^\ast$}

\title{Approximate fixed points and B-amenable groups}

\author{Jan Pachl  \\
        Toronto, Ontario, Canada}
\date{September 17, 2018}

\begin{document}
\maketitle

\begin{abstract}
A topological group $G$ is B-amenable if and only if every continuous affine action of $G$
on a bounded convex subset of a locally convex space has an approximate fixed point.
Similar results hold more generally for slightly uniformly continuous semigroup actions.
\end{abstract}


\section{Introduction}

For a locally compact group $G$,
denote by $\mathsf{L}_\infty (G)$ the space of (equivalence classes of)
bounded Haar-measurable functions,
by $\Cb(G)$ the space of bounded continuous functions on $G$,
and by $\Ub(\ru G)$ the space of bounded right uniformly continuous functions on $G$.
The following properties are equivalent characterizations of amenable
locally compact groups~\cite{Berglund1989aos}\cite{Greenleaf1969imt}%
\cite{Grigorchuk2017aep}\cite{Paterson1988ame}.
\begin{enumerate}[label=(\Alph*)]
\item\label{intro-Ub}
There exists a left-invariant mean on $\Ub(\ru G)$.
\item\label{intro-Cb}
There exists a left-invariant mean on $\Cb(G)$.
\item\label{intro-Linfty}
There exists a left-invariant mean on $\mathsf{L}_\infty (G)$.
\item\label{intro-FP}
Every continuous affine action of $G$ on a compact convex set has a fixed point.
\end{enumerate}
Properties~\ref{intro-Ub} to~\ref{intro-FP} have been extended in many directions,
including these two:
\begin{itemize}
\item[1.]
Expand the scope from locally compact groups to general topological groups.
Then $\mathsf{L}_\infty (G)$ is no longer available,
but $\Cb(G)$ and $\Ub(\ru G)$ are well-defined.
Conditions~\ref{intro-Ub} and~\ref{intro-FP} are still equivalent
but~\ref{intro-Ub} and~\ref{intro-Cb} are not~\cite{Grigorchuk2017aep}:
The infinite symmetric group with the topology of pointwise convergence
and the unitary group of a separable infinite dimensional Hilbert space
with the strong operator topology have property~\ref{intro-Ub} but not~\ref{intro-Cb}.
Following~\cite{Grigorchuk2017aep}, a topological group $G$ is defined to be \emph{amenable}
iff it satisfies property~\ref{intro-Ub}, and \emph{B-amenable} iff it satisfies~\ref{intro-Cb}.
\item[2.]
Consider affine actions on more general convex sets.
Here fixed points are too much to ask for.
The next best thing are approximate fixed points,
with property~\ref{intro-FP} replaced by
\end{itemize}
\begin{enumerate}[label=(\Alph*),resume]
\item
    \label{intro-AFP}
Every continuous affine action of $G$ on a bounded convex set has an approximate fixed point.
\end{enumerate}
\begin{itemize}
\item[]
Barroso, Mbombo and Pestov~\cite{Barroso2017tga} studied property~\ref{intro-AFP}.
They proved several results relating~\ref{intro-Ub} to the existence of approximate fixed points,
and showed that the infinite symmetric group does not have property~\ref{intro-AFP}.
\end{itemize}

\noindent
By Theorem~\ref{th:Bamenable} of the current paper,
\ref{intro-Cb} is equivalent to~\ref{intro-AFP} for general topological groups.
The proof is not specific to continuous actions by topological groups.
It applies equally well to semigroup actions that satisfy a certain
quite weak uniform continuity property;
that is presented in section~\ref{sec:AFP},
after definitions and preliminary results in section~\ref{sec:prelim}.
The general result covers also the case of amenable groups,
due to Schneider and Thom~\cite{Schneider2018ofs},
which is derived in section~\ref{sec:amenable}.


\section{Preliminaries}
    \label{sec:prelim}

Throughout the paper, linear spaces are over the scalar field,
either the reals or the complex numbers.
Functions are mappings to the scalar field.
Most of the following notation and terminology is taken from~\cite{Pachl2013usm}.
Several prerequisites are proved in~\cite{Pachl2013usm} for real-valued functions;
the complex case follows simply by considering the real and imaginary part of the function.

Topological, uniform and locally convex spaces are assumed to be Hausdorff.
When $S$ is a uniform space, $\Ub(S)$ denotes the space of bounded uniformly continuous
functions on $S$ with the sup norm.
When $G$ is a topological group, $\ru G$ is the right uniform space of $G$;
that is, the set $G$ with the right uniform structure of $G$.
For any completely regular topological space $T$, $\tfine T$ denotes the set $T$
with the fine uniform structure~\cite[1.25]{Pachl2013usm}.
Thus $\Ub(\tfine T)=\Cb(T)$, the space of bounded continuous functions on $T$.
When $S$ is a set with the discrete uniformity, $\Ub(S)$ is simply $\ellinfty(S)$.

When $X$ is a locally convex space, $X^\ast$ denotes its topological dual.
In particular, when $S$ is a~uniform space,
$\Ub(S)^\ast$ is the dual of the Banach space $\Ub(S)$.

Let $S$ be a semigroup, $s\iin S$, and let $f$ be a function on $S$.
Define the functions $_s f$ and $f_s$ (the \emph{left} and \emph{right translate of $f$})
by $_s f (t) := f(st)$ and $f_s (t) := f(ts)$ for $t\iin S$.

I need a short name for a semigroup $S$ with a uniform stucture on $S$
such that $_s f\iin\Ub(S)$ and $f_s \iin\Ub(S)$ for all $s\iin S$, $f\iin\Ub(S)$.
Such a structure will be called an \emph{FIU semigroup},
or an \emph{FIU group} when $S$ is actually a group.
(FIU stands for Functionally Invariant Uniform structure.)

\begin{examples} \mbox{}
\begin{enumerate}[label=(\alph*)]
\item
Let $G$ be a topological group.
Then $G$ with any one of the left, right, upper and lower uniformities~\cite{Roelcke1981ust}
is an FIU group.
\item
Let $S$ be a semitopological semigroup; that is, a Hausdorff topological space with a separately
continuous semigroup operation.
Then $S$ with its fine uniformity is an FIU semigroup.
In particular, any topological group with its fine uniformity is an FIU group.
\item
Every semiuniform semigroup~\cite[3.14]{Pachl2013usm} is an FUI semigroup.
That includes, for example, the completion of any topological group with its right uniformity,
and the unit ball of any Banach algebra with the multiplication operation and the metrizable
uniform structure defined by the norm.
\item
Let $S$ be a semigroup, and let $\Aset$ be a linear subspace of $\ellinfty(S)$
that separates the points of $S$
and is invariant under left and right translations by elements of $S$.
Equip $S$ with the uniformity induced by $\Aset$;
that is, the coarsest uniformity for which $\Aset\subseteq\Ub(S)$.
Then $S$ is a precompact FIU semigroup.
\qed
\end{enumerate}
\end{examples}

When $S$ is an FIU semigroup and $C$ is a subset of a locally convex space $X$,
$S\action C$ means that $S$ acts on $C$ by a mapping $(s,x)\mapsto s\act x$
from $S\times C$ to $C$.
The action $S\action C$ is said to be \emph{affine} iff the mapping $x\mapsto s\act x$
is affine for every $s\iin S$.
The action is \emph{jointly} or \emph{separately continuous}
iff the mapping $(s,x)\mapsto s\act x$ is such.
The action is \emph{slightly continuous}~\cite{Day1964cfp} iff there is at least one $x\iin C$
such that the mapping $s\mapsto s\act x$ is continuous from $S$ to $X$.
More generally, the action is \emph{slightly uniformly continuous} iff there is at least one $x\iin C$
such that the mapping $s\mapsto s\act x$ is uniformly continuous from $S$ to $X$.
Obviously every separately continuous action is slightly continuous.
An \emph{approximate fixed point} for the action $S\action C$ is a net
$\{x_\gamma\}_\gamma$ in $C$ such that $\lim_\gamma x_\gamma - s\act x_\gamma = 0$
in the topology of $X$ for every $s\iin S$.

For $s\iin S$, the \emph{point mass $\pmass(s)\iin\Ub(S)^\ast$ at $s$} is defined by
$\pmass(s)(f):=f(s)$ for $f\iin\Ub(S)$.
The subspace of $\Ub(S)^\ast$ consisting of finite linear combinations of point masses
(called \emph{molecular measures}) is denoted by $\Mol(S)$.
The \emph{\wstar\ topology} on $\Mol(S)$ is the restriction of the \wstar\ topology
(the $\Ub(S)$-weak topology) on $\Ub(S)^\ast$.

For $s\iin S$ and $\mmeas\iin\Ub(S)^\ast$, define $s\conv\mmeas\iin\Ub(S)^\ast$ by
$s\conv\mmeas (f) := \mmeas( _s f)$ for $f\iin\Ub(S)$.
Clearly $s\conv\pmass(t) = \pmass(st)$ for $s,t\iin S$.
Say that $\mmeas\iin\Ub(S)^\ast$ is \emph{left-invariant} iff
$s\conv\mmeas = \mmeas$ for all $s\iin S$.

A functional $\mmeas\iin\Ub(S)^\ast$ is a \emph{mean} iff $\mmeas\geq 0$ and $\mmeas(1) = 1$.
Denote by $\Means$ the set of means in $\Ub(S)^\ast$.
In the \wstar\ topology on $\Ub(S)^\ast$, $\Means$ is a compact convex set.
Write $\Mol^{+1}(S):=\Means\cap\Mol(S)$.
Every $\mmeas\iin\Mol^{+1}(S)$ is of the form $\mmeas=\sum_{i=0}^j r_i \pmass(s_i)$
where $0\leq r_i \leq 1$ and $s_i\iin S$ for $i=0,1,\dots,j$, and $\sum_{i=0}^j r_i = 1$.

When $S$ is a uniform space,
a subset of $\Ub(S)$ is a UEB$(S)$ \emph{set}
iff it is uniformly equicontinuous and bounded in the sup norm~\cite[1.19]{Pachl2013usm}.
The UEB$(S)$ \emph{topology} on $\Ub(S)^\ast$
is the topology of uniform convergence on UEB$(S)$ subsets of $\Ub(S)$.
In the particular case $S=\tfine T$, where $T$ is a completely regular space,
the UEB$(\tfine T)$ topology on $\Cb(T)^\ast$ is also called the EB$(T)$ \emph{topology};
it is the topology of uniform convergence on equicontinuous uniformly bounded subsets of $\Cb(T)$.

\begin{lemma}
    \label{lem:prelim}
Let $S$ be an FIU semigroup.
\begin{enumerate}[label=(\alph*)]
\item\label{lem:prelim:EBdual}
When $\Mol(S)$ is equipped with the UEB$(S)$ topology, the dual of $\Mol(S)$ is $\Ub(S)$.
\item\label{lem:prelim:dense}
The set $\Mol^{+1}(S)$ is \wstar\ dense in $\Means$.
\item\label{lem:prelim:leftcon}
Let $s\iin S$.
Equip $\Ub(S)^\ast$ with the \wstar\ topology.
Then the mapping $\mmeas\mapsto s\conv\mmeas$ is continuous from $\Ub(S)^\ast$ to itself.
\item\label{lem:prelim:rightcon}
Let $\mmeas\iin\Mol(S)$.
Equip $\Mol(S)$ with the \wstar\ topology.
Then the mapping $s\mapsto s\conv\mmeas$ is uniformly continuous from $S$ to $\Mol(S)$.
\item\label{lem:prelim:wstarext}
Let $X$ be a locally convex space.
Let $\Phi\colon S \to X$ be a uniformly continuous mapping whose range $\Phi(S)$ is bounded in $X$.
Denote by $\widetilde{\Phi} \colon \Mol(S) \to X$ the unique linear mapping
such that $\widetilde{\Phi} \circ \pmass = \Phi$.
Then $\widetilde{\Phi}$ is continuous from $\Mol(S)$ with the \wstar\ topology
to $X$ with its weak topology.
\end{enumerate}
\end{lemma}

\begin{proof}
\ref{lem:prelim:EBdual}
(See Lemma~6.5 in~\cite{Pachl2013usm}.)
If $\Fset\subseteq\Ub(S)$ is a UEB$(S)$ set such that $\norm{f}\leq 1$ for $f\iin\Fset$ then
\[
\psm(s,t) := \sup \{ \abs{f(s)-f(t)} \mid f\iin\Fset \}, \quad s,t\iin S,
\]
defines a uniformly continuous pseudometric $\psm$ on $S$, and $\Fset$ is a subset of
\[
\BLipb(\psm) := \{ f\iin\Ub(S) \mid \norm{f}\leq 1 \text{ and } \abs{f(s)-f(t)}\leq\psm(s,t)
                                    \text{ for } s,t\in S \}.
\]
Thus the UEB$(S)$ topology is the topology of uniform convergence on the sets $\BLipb(\psm)$ for
uniformly continuous pseudometrics $\psm$.
Each such set is compact in the $\Mol(S)$-weak topology on $\Ub(S)$.
Apply the Mackey--Arens theorem~\cite[IV.3.2]{Schaefer1971tvs}.

\ref{lem:prelim:dense}
is a simple application of the Hahn--Banach theorem:
The set $\Mol^{+1}(S)$ is convex.
If a real-valued $f\iin\Ub(S)$ and a real number $r$ are such that
$\mmeas(f)\leq r$ for all $\mmeas\iin\Mol^{+1}(S)$ then in particular
$f(s)\leq r$ for all $s\iin S$,
therefore $\mmeas(f)\leq r$ for all $\mmeas\iin\Means$.

\ref{lem:prelim:leftcon}
follows from the definition of $s\conv\mmeas$.

\ref{lem:prelim:rightcon}
Write $\mmeas=\sum_{i=0}^j r_i \pmass(s_i)$ and fix $f\iin\Ub(S)$.
For $s,t\iin S$ we have
\[
\abs{s\conv\mmeas(f) - t\conv\mmeas(f)}
= \abs{\sum_{i=0}^j r_i (s\conv\pmass(s_i)(f) - t\conv\pmass(s_i)(f))}
\leq \sum_{i=0}^j \abs{r_i} \cdot \abs{f_{s_i}(s) - f_{s_i}(t)} .
\]
Since $f_{s_i}\iin\Ub(S)$ for $i=0,1,\dots,j$, it follows that the function
$s\mapsto s\conv\mmeas(f)$ is uniformly continuous.

\ref{lem:prelim:wstarext}
Take any $\xi\iin X^\ast$.
Then $\xi\circ\Phi$ is uniformly continuous, and it is bounded because $\Phi(S)$ is bounded in $X$.
Thus $\xi\circ\Phi\iin\Ub(S)$.
From the linearity of $\widetilde{\Phi}$ we get
$\xi(\widetilde{\Phi}(\mmeas)) = \mmeas(\xi\circ\Phi)$ for all $\mmeas\iin\Mol(S)$.
\end{proof}


\section{Approximate fixed point property}
    \label{sec:AFP}

Let $C$ be a convex subset of a locally convex space $X$.
Following Barroso et al~\cite{Barroso2017tga},
say that a topological group $G$ has the
\emph{approximate fixed point (AFP) property on $C$}
iff every jointly continuous affine action of $G$ on $C$ has an approximate fixed point.

When $C$ is compact, the AFP property is equivalent to the existence of fixed points.
Indeed, if $\{x_\gamma\}_\gamma$ is an approximate fixed point for $G\action C$
and if the mapping $x\mapsto g\act x$ is continuous on $C$ for every $g\iin G$
then every cluster point of the net $\{x_\gamma\}_\gamma$ in $C$ is a fixed point.

The following theorem,
an abstract version of
the ``celebrated method of Day''~\cite{Day1957as}\cite[\S2.4]{Greenleaf1969imt},
states that an affine action has an approximate fixed point
if it has a weak approximate fixed point.

\begin{theorem}
    \label{th:Day}
Let $S$ be a semigroup, and let $S\action C$ be an affine action of $S$ on a convex subset
of a~locally convex space $X$.
Assume there is a net $\{x_\gamma\}_\gamma$ in $C$
such that $\lim_\gamma x_\gamma - s\act x_\gamma = 0$
in the weak topology of $X$ for every $s\iin S$.
Then the action has an approximate fixed point.
\end{theorem}

\begin{proof}
Let $Y:=X^S$ be the space of all families $\{x_s\}_{s\in S}$ of $x_s \iin X$;
that is, the product of copies of $X$, one for each $s\iin S$.
Consider two locally convex topologies on $Y$:
$t$ is the product topology,
and $t_{w}$ is the product of the weak topologies on the copies of $X$.
The weak topology on the product of locally convex spaces
is the product of weak topologies~\cite[IV.4.3]{Schaefer1971tvs}.
Hence $t$ and $t_{w}$ on $Y$ have the same dual
and therefore the same closed convex sets~\cite[II.9.2]{Schaefer1971tvs}.
Define the convex set $D\subseteq Y$ by
$D := \{ \{x - s\act x\}_{s\in S} \mid x\in C \}$.
By the assumption, $0\iin Y$ is in the $t_{w}$ closure of $D$ which is also the $t$ closure of $D$,
so that the conclusion holds.
\end{proof}

Next we apply Theorem~\ref{th:Day} to the natural affine action $(s,\mmeas)\mapsto s\conv\mmeas$
of an FIU semigroup $S$ on $\Means$.
The set $\Mol^{+1}(S)\subseteq\Means$ is $S$-invariant under the action;
in fact, if $\mmeas=\sum_{i=0}^j r_i \pmass(s_i)$ then
$s\conv\mmeas=\sum_{i=0}^j r_i \pmass(s s_i)$.

The next theorem is essentially due to Day~\cite{Day1957as}\cite{Day1961fpt}\cite{Day1964cfp},
for whom approximate fixed point were a means for constructing fixed points.
The means in $\Mol(S)$ are the \emph{finite means} of Day.
The equivalence \ref{th:UbAmenable:i}$\Leftrightarrow$\ref{th:UbAmenable:iv} is a variant
of a fixed-point result mentioned by Berglund et al~\cite[2.3.30]{Berglund1989aos}
in the general setting of invariant means on a space $\Aset\subseteq\ellinfty(S)$.

\begin{theorem}
    \label{th:UbAmenable}
The following properties of an FIU semigroup $S$ are equivalent:
\begin{enumerate}[label=(\roman*)]
\item
    \label{th:UbAmenable:i}
There exists a left-invariant mean in $\Means$.
\item
    \label{th:UbAmenable:ii}
There exists a net $\{\mmeas_\gamma\}_\gamma$ in $\Mol^{+1}(S)$
such that for every $s \iin S$ we have $\lim_\gamma \mmeas_\gamma - s\conv\mmeas_\gamma = 0$
in the \wstar\ topology.
\item
    \label{th:UbAmenable:iii}
There exists a net $\{\mmeas_\gamma\}_\gamma$ in $\Mol^{+1}(S)$
such that for every $s \iin S$ we have $\lim_\gamma \mmeas_\gamma - s\conv\mmeas_\gamma = 0$
in the UEB$(S)$ topology.
\item
    \label{th:UbAmenable:iv}
Every slightly uniformly continuous affine action of $S$ on a bounded convex subset
of a locally convex space has an approximate fixed point.
\end{enumerate}
\end{theorem}

\begin{proof}
First let $\mmeas\in\Means$ be a left-invariant mean.
By \ref{lem:prelim}\ref{lem:prelim:dense} there is a net $\{\mmeas_\gamma\}_\gamma$
in $\Mol^{+1}(S)$ such that $\lim_\gamma \mmeas_\gamma =\mmeas$ in the \wstar\ topology.
Then $\lim_\gamma s\conv\mmeas_\gamma =s\conv\mmeas$ for every $s\iin S$
by~\ref{lem:prelim}\ref{lem:prelim:leftcon},
hence $\lim_\gamma \mmeas_\gamma - s\conv\mmeas_\gamma = 0$.
That proves \ref{th:UbAmenable:i}$\Rightarrow$\ref{th:UbAmenable:ii}.

To prove \ref{th:UbAmenable:ii}$\Rightarrow$\ref{th:UbAmenable:i},
let $\{\mmeas_\gamma\}_\gamma$ be a net in $\Mol^{+1}(S)$
such that $\lim_\gamma \mmeas_\gamma - s\conv\mmeas_\gamma = 0$
in the \wstar\ topology for every $s\iin S$.
The set $\Means$ is \wstar\ compact, hence the net $\{\mmeas_\gamma\}_\gamma$
has a cluster point $\mmeas\in\Means$.
Then $s\conv\mmeas = \mmeas$ for $s\iin S$ by~\ref{lem:prelim}\ref{lem:prelim:leftcon}.

Obviously \ref{th:UbAmenable:iii}$\Rightarrow$\ref{th:UbAmenable:ii}.
The implication \ref{th:UbAmenable:ii}$\Rightarrow$\ref{th:UbAmenable:iii} follows
from~\ref{lem:prelim}\ref{lem:prelim:EBdual} and Theorem~\ref{th:Day}.

The action $(s,\mmeas)\mapsto s\conv\mmeas$ of $S$ on the convex set $\Mol^{+1}(S)\subseteq\Mol(S)$
is slightly uniformly continuous when $\Mol(S)$ is given the \wstar\ topology.
Indeed, by~\ref{lem:prelim}\ref{lem:prelim:rightcon} for every $\mmeas\iin\Mol^{+1}(S)$ the mapping
$s\mapsto s\conv\mmeas$ is uniformly continuous.
Hence \ref{th:UbAmenable:iv}$\Rightarrow$\ref{th:UbAmenable:ii}.

To prove \ref{th:UbAmenable:ii}$\Rightarrow$\ref{th:UbAmenable:iv},
let $S\action C$ be a slightly uniformly continuous affine action
on a bounded convex subset $C$ of a locally convex space $X$.
There is $x\iin C$ for which the mapping $s\mapsto \Phi(s) := s\act x$ is uniformly continuous
from $S$ to $X$.
Let $\widetilde{\Phi} \colon \Mol(S) \to X$ be the mapping
from~\ref{lem:prelim}\ref{lem:prelim:wstarext}.
We have $s\act\Phi(t)=\Phi(st)$ for $s,t\iin S$,
hence $s\act\widetilde{\Phi}(\mmeas)=\widetilde{\Phi}(s\conv\mmeas)$ for $s\iin S$
and $\mmeas\iin\Mol^{+1}(S)$.
Let $\mmeas_\gamma$ be as in \ref{th:UbAmenable:ii}.
Set $x_\gamma := \widetilde{\Phi}(\mmeas_\gamma)$.
Then
$
x_\gamma - s\act x_\gamma
= \widetilde{\Phi}(\mmeas_\gamma) - \widetilde{\Phi}(s\conv\mmeas_\gamma)
= \widetilde{\Phi}(\mmeas_\gamma - s\conv\mmeas_\gamma)
$,
and by~\ref{lem:prelim}\ref{lem:prelim:wstarext} we get
$\lim_\gamma x_\gamma - s\act x_\gamma = 0$ in the weak topology of $X$ for every $s\iin S$.
Hence \ref{th:UbAmenable:iv} follows by Theorem~\ref{th:Day}.
\end{proof}


\section{Amenable groups}
    \label{sec:amenable}

When $G$ is a topological group, $\ru G$ is an FIU group.
The group $G$ is said to be \emph{amenable}~\cite{Grigorchuk2017aep}
iff there is a left-invariant mean on $\Ub(\ru G)$.
The following instance of Theorem~\ref{th:UbAmenable}
is due to Schneider and Thom (Theorem~3.2 and Corollary~4.8 in~\cite{Schneider2018ofs}).

\begin{theorem}
    \label{th:amenable}
The following properties of a topological group $G$ are equivalent:
\begin{enumerate}[label=(\roman*)]
\item
    \label{th:amenable:i}
$G$ is amenable.
\item
    \label{th:amenable:ii}
There exists a net $\{\mmeas_\gamma\}_\gamma$ in $\Mol^{+1}(G)$
such that for every $g\iin G$ we have $\lim_\gamma \mmeas_\gamma - g\conv\mmeas_\gamma = 0$
in the $\Ub(\ru G)$-weak topology.
\item
    \label{th:amenable:iii}
There exists a net $\{\mmeas_\gamma\}_\gamma$ in $\Mol^{+1}(G)$
such that for every $g\iin G$ we have $\lim_\gamma \mmeas_\gamma - g\conv\mmeas_\gamma = 0$
in the UEB$(\ru G)$ topology.
\item
    \label{th:amenable:iv}
Every slightly uniformly continuous affine action of $\ru G$ on a bounded convex subset
of a locally convex space has an approximate fixed point.
\end{enumerate}
\end{theorem}


\section{B-amenable groups}
    \label{sec:Bamenable}

When $G$ is a topological group, $\tfine G$ is an FIU group.
The group $G$ is said to be \emph{B-amenable}~\cite{Grigorchuk2017aep}
iff there is a left-invariant mean on $\Cb(G)=\Ub(\tfine G)$.

\begin{theorem}
    \label{th:Bamenable}
The following properties of a topological group $G$ are equivalent:
\begin{enumerate}[label=(\roman*)]
\item
    \label{th:Bamenable:i}
$G$ is B-amenable.
\item
    \label{th:Bamenable:ii}
There exists a net $\{\mmeas_\gamma\}_\gamma$ in $\Mol^{+1}(G)$
such that for every $g \iin G$ we have $\lim_\gamma \mmeas_\gamma - g\conv\mmeas_\gamma = 0$
in the $\Cb(G)$-weak topology.
\item
    \label{th:Bamenable:iii}
There exists a net $\{\mmeas_\gamma\}_\gamma$ in $\Mol^{+1}(G)$
such that for every $g \iin G$ we have $\lim_\gamma \mmeas_\gamma - g\conv\mmeas_\gamma = 0$
in the EB$(G)$ topology.
\item
    \label{th:Bamenable:iv}
Every slightly continuous affine action of $G$ on a bounded convex subset of a locally convex space
has an approximate fixed point.
\item
    \label{th:Bamenable:v}
Every separately continuous affine action of $G$ on a bounded convex subset of a locally convex space
has an approximate fixed point.
\item
    \label{th:Bamenable:vi}
Every jointly continuous affine action of $G$ on a bounded convex subset of a locally convex space
has an approximate fixed point.
\end{enumerate}
\end{theorem}

\begin{proof}
Properties \ref{th:Bamenable:i}, \ref{th:Bamenable:ii},
\ref{th:Bamenable:iii} and \ref{th:Bamenable:iv}
are equivalent by Theorem~\ref{th:UbAmenable}.
Obviously \ref{th:Bamenable:iv}$\Rightarrow$\ref{th:Bamenable:v}$\Rightarrow$\ref{th:Bamenable:vi}.

To prove \ref{th:Bamenable:vi}$\Rightarrow$\ref{th:Bamenable:ii},
it is enough to show that the action $\conv$ of $G$ on $\Mol^{+1}(G)$ is jointly continuous
when $\Mol^{+1}(G)$ is equipped with the $\Cb(G)$-weak topology.
By~\cite[5.16]{Pachl2013usm},
the $\Cb(G)$-weak topology and the $\Ub(\ru G)$-weak topology coincide on $\Mol^{+1}(G)$.
It follows that the action is jointly continuous by~\cite[9.36-3]{Pachl2013usm}.
\end{proof}

In the terminology of~\cite{Barroso2017tga},
\ref{th:Bamenable:vi} states that $G$ has the AFP property on every bounded
convex subset of a locally convex space.


\section{Concluding remarks}

The equivalence
\ref{th:Bamenable:iv}$\Leftrightarrow$\ref{th:Bamenable:v}$\Leftrightarrow$\ref{th:Bamenable:vi}
in Theorem~\ref{th:Bamenable} stands in contrast to fixed-point properties for affine actions
on compact convex sets~\cite{Day1961fpt}\cite{Day1964cfp}\cite{Grigorchuk2017aep}:
The infinite symmetric group has the fixed-point property for separately continuous actions
but not for slightly continuous ones.
I do not know if there is a topological group that has the fixed-point property for affine
jointly continuous actions but not for separately continuous ones.

Barroso et al~\cite[3.4]{Barroso2017tga} prove that every amenable locally compact group
has the AFP property on every complete convex bounded subset of a locally convex space,
and ask if the same holds without the \emph{complete} restriction.
The answer is yes by Theorem~\ref{th:Bamenable}, because
every amenable locally compact group is B-amenable~\cite[3.10]{Grigorchuk2017aep}. \\

\textbf{Acknowledgment.}
I wish to thank Pierre de la Harpe and Alf Onshuus Ni\~{n}o for helpful comments.



\begin{thebibliography}{10}

\bibitem{Barroso2017tga}
Barroso, C.~S., Mbombo, B.~R., and Pestov, V.~G.
\newblock \textit{ On topological groups with an approximate fixed point property\/}.
\newblock An. Acad. Brasil. Ci\^{e}nc. \textbf{89}, 1 (2017), 19--30.

\bibitem{Berglund1989aos}
Berglund, J.~F., Junghenn, H.~D., and Milnes, P.
\newblock \textit{ Analysis on semigroups}.
\newblock John Wiley \& Sons Inc., New York, 1989.

\bibitem{Day1957as}
Day, M.~M.
\newblock \textit{ Amenable semigroups\/}.
\newblock Illinois J. Math. \textbf{1} (1957), 509--544.

\bibitem{Day1961fpt}
Day, M.~M.
\newblock \textit{ Fixed-point theorems for compact convex sets\/}.
\newblock Illinois J. Math. \textbf{5} (1961), 585--590.

\bibitem{Day1964cfp}
Day, M.~M.
\newblock \textit{ Correction to my paper ``{F}ixed-point theorems for compact
  convex sets''\/}.
\newblock Illinois J. Math. \textbf{8} (1964), 713.

\bibitem{Greenleaf1969imt}
Greenleaf, F.~P.
\newblock \textit{ Invariant means on topological groups and their
  applications}.
\newblock Van Nostrand Mathematical Studies, No. 16. Van Nostrand Reinhold Co.,
  New York-Toronto-London, 1969.

\bibitem{Grigorchuk2017aep}
Grigorchuk, R., and de~la Harpe, P.
\newblock \textit{ Amenability and ergodic properties of topological groups:
  from {B}ogolyubov onwards\/}.
\newblock Groups, graphs and random walks, London Math. Soc. Lecture Note Ser.,
  Vol.~436. Cambridge Univ. Press, Cambridge, 2017, pp.~215--249.

\bibitem{Pachl2013usm}
Pachl, J.
\newblock \textit{ Uniform spaces and measures}, Fields Institute Monographs,
  Vol.~30.
\newblock Springer, New York, 2013.
Corrections and supplements:
{www.fields.utoronto.ca/publications/supplements}

\bibitem{Paterson1988ame}
Paterson, A. L.~T.
\newblock \textit{ Amenability}, Mathematical Surveys and Monographs, Vol.~29.
\newblock American Mathematical Society, Providence, RI, 1988.

\bibitem{Roelcke1981ust}
Roelcke, W., and Dierolf, S.
\newblock \textit{ Uniform structures on topological groups and their quotients}.
\newblock McGraw-Hill, New York, 1981.

\bibitem{Schaefer1971tvs}
Schaefer, H.~H.
\newblock \textit{ Topological vector spaces}.
\newblock Springer-Verlag, New York, 1971.
\newblock Third printing corrected.

\bibitem{Schneider2018ofs}
Schneider, F.~M., and Thom, A.
\newblock \textit{ On {F}\o lner sets in topological groups\/}.
\newblock Compos. Math. \textbf{154}, 7 (2018), 1333--1361.

\end{thebibliography}
\end{document}